\documentclass[a4paper,
fontsize=11pt,%
oneside,%
numbers=enddot]{scrartcl}
\KOMAoptions{DIV=12}    
\usepackage[T1]{fontenc}
\usepackage{textcomp}
\usepackage[fleqn]{amsmath}
\usepackage{amsthm}
\usepackage{amssymb}
\usepackage{amsfonts}
\usepackage{libertine}
\usepackage[libertine,
slantedGreek,
nosymbolsc,
nonewtxmathopt,
subscriptcorrection]{newtxmath}
\usepackage[scaled=0.95,varqu,varl]{inconsolata}
\usepackage[scr=boondox]{mathalpha}   
\usepackage{euscript}   
\usepackage{leftindex}
\allowdisplaybreaks[1]  
\numberwithin{equation}{section}    
\usepackage[svgnames,hyperref]{xcolor}
\definecolor{dblue}{HTML}{0455BF}
\definecolor{dgreen}{HTML}{02724A}
\definecolor{dgreen2}{HTML}{025951}
\definecolor{dred}{HTML}{D90404}
\definecolor{dviolet}{HTML}{42208C}
\definecolor{labelkey}{HTML}{025951}
\definecolor{refkey}{HTML}{025951}
\addtokomafont{section}{\centering}

\usepackage{enumitem}
\setlist{itemsep=-2.0pt}

\usepackage{upref}  
\usepackage{hyperref}
\hypersetup{colorlinks=true,
linktocpage=true,
linkcolor=dblue,
citecolor=dgreen,
urlcolor=dred,
pdfencoding=auto,
hypertexnames=false}
\makeatletter
\g@addto@macro\th@plain{
\thm@headfont{\bfseries\sffamily}
\thm@notefont{}}
\g@addto@macro\th@definition{
\thm@headfont{\bfseries\sffamily}
\thm@notefont{}}
\g@addto@macro\th@remark{
\thm@headfont{\bfseries\sffamily}
\thm@notefont{}}
\makeatother
\theoremstyle{plain}
\newtheorem{theorem}{Th\'eor\`eme}[section]
\newtheorem{proposition}[theorem]{Proposition}
\newtheorem{corollary}[theorem]{Corollaire}
\newtheorem{lemma}[theorem]{Lemme}

\theoremstyle{definition}
\newtheorem{definition}[theorem]{D\'efinition}
\newtheorem{example}[theorem]{Exemple}
\newtheorem{problem}[theorem]{Probl\`eme}

\theoremstyle{remark}
\newtheorem{remark}[theorem]{Remarque}

\usepackage[extdef=true]{delimset}
\DeclareMathDelimiterSet{\scal}[2]{
\selectdelim[l]<{#1}
\mathpunct{}\selectdelim[p]|
{#2}\selectdelim[r]>}

\newcommand{\menge}[2]{\bigl\{{#1}\mid{#2}\bigr\}} 
\DeclareMathDelimiterSet{\Menge}[2]{\selectdelim[l]\{
{#1}\selectdelim[m]|{#2}\selectdelim[r]\}}

\makeatletter
\def\upintkern@{\mkern-7mu\mathchoice{\mkern-3.5mu}{}{}{}}
\def\upintdots@{\mathchoice{\mkern-4mu\@cdots\mkern-4mu}%
{{\cdotp}\mkern1.5mu{\cdotp}\mkern1.5mu{\cdotp}}%
{{\cdotp}\mkern1mu{\cdotp}\mkern1mu{\cdotp}}%
{{\cdotp}\mkern1mu{\cdotp}\mkern1mu{\cdotp}}}
\makeatother
\DeclareFontFamily{OMX}{mdbch}{}
\DeclareFontShape{OMX}{mdbch}{m}{n}{ <->s * [0.8]  mdbchr7v }{}
\DeclareFontShape{OMX}{mdbch}{b}{n}{ <->s * [0.8]  mdbchb7v }{}
\DeclareFontShape{OMX}{mdbch}{bx}{n}{<->ssub * mdbch/b/n}{}
\DeclareSymbolFont{uplargesymbols}{OMX}{mdbch}{m}{n}
\SetSymbolFont{uplargesymbols}{bold}{OMX}{mdbch}{b}{n}
\DeclareMathSymbol{\upintop}{\mathop}{uplargesymbols}{82}
\DeclareMathSymbol{\upointop}{\mathop}{uplargesymbols}{"48}
\makeatletter
\renewcommand{\int}{\DOTSI\upintop\ilimits@}
\renewcommand{\oint}{\DOTSI\upointop\ilimits@}
\makeatother

\newcommand{\RR}{\mathbb{R}}
\newcommand{\NN}{\mathbb{N}}
\newcommand{\HH}{\mathcal{H}}
\newcommand{\HHH}{\ensuremath{\boldsymbol{{\mathsf{H}}}}}

\newcommand{\GG}{\mathcal{G}}
\newcommand{\pinf}{{+}\infty}
\newcommand{\minf}{{-}\infty}

\newcommand{\RX}{\intv[l]0{\minf}{\pinf}}

\newcommand{\RPP}{\intv[o]0{0}{\pinf}}

\newcommand{\emp}{\varnothing}

\newcommand{\minimize}[2]{\underset{\substack{{#1}}}
{\operatorname{minimiser}}\;\;#2}

\newcommand{\pushfwd}%
{\ensuremath{\mbox{\Large$\,\triangleright\,$}}}

\DeclareMathOperator{\Argmin}{Argmin}

\newcommand{\Id}{\mathrm{Id}}
\newcommand{\moyo}[2]{\leftindex[I]^{#2}{#1}}

\DeclareMathOperator{\dom}{dom}

\DeclareMathOperator{\prox}{prox}
\DeclareMathOperator{\soft}{doux}
\DeclareMathOperator{\proj}{proj}

\DeclareFontFamily{U}{mathb}{}
\DeclareFontShape{U}{mathb}{m}{n}{<-5.5> mathb5 <5.5-6.5> mathb6 
<6.5-7.5> mathb7 <7.5-8.5> mathb8 <8.5-9.5> mathb9 <9.5-11> mathb10
<11-> mathb12}{}
\DeclareSymbolFont{mathb}{U}{mathb}{m}{n}
\DeclareFontSubstitution{U}{mathb}{m}{n}
\DeclareMathSymbol{\blackdiamond}{\mathbin}{mathb}{"0C}

\renewcommand{\leq}{\leqslant}
\renewcommand{\geq}{\geqslant}

\newenvironment{resume}{%
\vspace*{-0.50cm}
\small
\quotation%
\noindent%
{\normalfont\bfseries\sffamily
\nobreak{R\'esum\'e.}\ }%
}{%
\endquotation%
\medskip
}

\renewenvironment{abstract}{%
\vspace*{-0.50cm}
\small
\quotation%
\noindent%
{\normalfont\bfseries\sffamily
\nobreak\abstractname\ }%
}{%
\endquotation%
\medskip
}
\renewcommand{\abstractname}{Abstract.}
\newcommand\keywordsname{Keywords.}

\usepackage[auth-sc]{authblk}
\newcommand{\email}[1]{\href{mailto:#1}{\nolinkurl{#1}}}
\renewcommand*\Affilfont{\normalfont\normalsize}
\newcommand\affilcr{\protect\\ \protect\Affilfont}
\makeatletter
\renewcommand\AB@affilsepx{\protect\\[0.5em]}
\makeatother

\author[1]{Patrick L. Combettes}
\affil[1]{North Carolina State University
\affilcr
Department of Mathematics
\affilcr
Raleigh, NC 27695, USA
\affilcr
\email{plc@math.ncsu.edu}
}

\begin{document}

\title{%
The Proximal Gradient Method\thanks{%
P. L. Combettes.
Courriel : \email{plc@math.ncsu.edu}.
T\'el\'ephone : +1 919 515 2671.
Ce travail a \'et\'e subventionn\'e par le contrat
CCF-2211123 de la National Science Foundation.
}}

\date{~}

\maketitle

\vskip -15mm

\begin{abstract}
The proximal gradient method is a splitting algorithm for the
minimization of the sum of two convex functions, one of which is
smooth. It has applications in areas such as mechanics,
inverse problems, machine learning, image reconstruction,
variational inequalities, statistics, operations research, and
optimal transportation. Its formalism encompasses a wide variety 
of numerical methods in optimization such as gradient descent,
projected gradient, iterative thresholding, alternating
projections, the constrained Landweber method, as well as various
algorithms in statistics and sparse data analysis. This
paper aims at providing an account of the main properties of the
proximal gradient method and to discuss some of its applications.
\end{abstract}

{%
\begin{center}
{\huge\sffamily\bfseries La M\'ethode du Gradient Proxim\'e}\\[6mm]
{\scshape\Large Patrick L. Combettes}\\[3mm]
North Carolina State University\\
Department of Mathematics\\
Raleigh, NC 27695, USA\\
Courriel : \email{plc@math.ncsu.edu}.
\end{center}

\begin{resume}
La m\'ethode du gradient proxim\'e est un algorithme d'\'eclatement
pour la minimisation de la somme de deux fonctions convexes, dont
l'une est lisse. Elle trouve des applications dans des domaines
tels que la m\'ecanique, le traitement du signal, les probl\`emes
inverses, l'apprentissage automatique, la reconstruction d'images,
les in\'equations variationnelles, les statistiques, la recherche
op\'erationnelle et le transport optimal. Son formalisme englobe
une grande vari\'et\'e de m\'ethodes num\'eriques en optimisation,
telles que la descente de gradient, le gradient projet\'e, la
m\'ethode de seuillage it\'eratif, la m\'ethode des projections
altern\'ees, la m\'ethode de Landweber contrainte, ainsi que
divers algorithmes en statistique et en analyse parcimonieuse de
donn\'ees. Cette synth\`ese vise \`a donner un aper\c{c}u des
principales propri\'et\'es de la m\'ethode du gradient proxim\'e et
d'aborder certaines de ses applications. 
\end{resume}

\vspace{12mm}
\newpage

\section{Introduction}

\noindent
{\bfseries Notations.}
$\HH$, $\GG$ et $\GG_k$ d\'esignent des espaces euclidiens, \`a
savoir des espaces hilbertiens r\'eels de dimension finie. On
note $\scal{\cdot}{\cdot}$ leur produit scalaire et $\|\cdot\|$ la
norme associ\'ee. Une fonction $f\colon\HH\to\RX$ est \emph{propre}
si $\dom f=\menge{x\in\HH}{f(x)<\pinf}\neq\emp$. La classe des 
fonctions semicontinues inf\'erieurement, convexes et propres de 
$\HH$ dans $\RX$ se note $\Gamma_0(\HH)$. Enfin, 
$\text{ir}\,C$ d\'esigne l'int\'erieur relatif d'une partie convexe 
$C\subset\HH$, \emph{i.e.}, son int\'erieur en tant que
sous-ensemble du plus petit espace affine qui le contient. 

Le th\`eme central de cet article est le probl\`eme de minimisation
convexe suivant, qui sous-tend une multitude de formulations
variationnelles en math\'ematiques appliqu\'ees et dans les
sciences de l'ing\'enieur.

\begin{problem}
\label{prob:1}
Soient $\beta\in\RPP$, $f\in\Gamma_0(\HH)$ et $g\colon\HH\to\RR$
une fonction convexe diff\'erentiable dont le gradient $\nabla g$
est $\beta$-lipschitzien. L'objectif est de 
\begin{equation}
\label{e:prob1}
\minimize{x\in\HH}{f(x)+g(x)}
\end{equation}
sous l'hypoth\`ese que l'ensemble des solutions $\Argmin(f+g)$ est
non vide.
\end{problem}

La m\'ethode du gradient proxim\'e s'inscrit dans la classe des
m\'ethodes d'\'eclatement, qui visent \`a d\'ecomposer un
probl\`eme en composantes \'el\'ementaires faciles \`a activer
individuellement dans un algorithme \cite{Acnu24,Lion70}. Dans le
cas du Probl\`eme~\ref{prob:1}, il s'agit d'activer la fonction
$f$ via son op\'erateur de proximit\'e et la fonction $g$, qui est
lisse, via son gradient. Ainsi, la m\'ethode du \emph{gradient
proxim\'e} alterne un pas de gradient sur $g$ et un pas proximal
sur $f$.

On fournit dans la Section~\ref{sec:2} quelques r\'esultats
essentiels sur les fonctions convexes. La m\'ethode du
gradient proxim\'e est d\'ecrite dans la Section~\ref{sec:3}, qui
couvre \'egalement ses propri\'et\'es asymptotiques essentielles.
Diverses applications de la m\'ethode sont d\'ecrites dans la
Section~\ref{sec:4}. La Section~\ref{sec:5} sur la version duale de
la m\'ethode du gradient proxim\'e et la Section~\ref{sec:6} sur sa
version multivari\'ee, ouvrent de nouveaux champs d'application.
On r\'ecapitule dans la Section~\ref{sec:7} les points principaux
de cette synth\`ese. 

\section{Cadre math\'ematique}
\label{sec:2}

Les r\'esultats suivants concernent l'op\'erateur de proximit\'e de
Moreau \cite{Mor62b,More65}, un outil essentiel en analyse convexe
non lisse.

\begin{lemma}
\label{l:1}
Soient $f\in\Gamma_0(\HH)$ et $x\in\HH$. Posons
\begin{equation}
\psi\colon\HH\to\RX\colon y\mapsto f(y)+\dfrac{1}{2}\|x-y\|^2.
\end{equation}
Alors $\psi$ admet un minimiseur unique not\'e $\prox_fx$ et
appel\'e le \emph{point proximal} de $x$ relativement \`a $f$. Il
est caract\'eris\'e comme suit :
\begin{equation}
\label{e:p07}
(\forall p\in\HH)\quad\bigl[\;p=\prox_fx\;\;\Leftrightarrow\;\;
(\forall y\in\HH)\quad\scal{y-p}{x-p}+f(p)\leq f(y)\;\bigr].
\end{equation}
On appelle $\prox_f$ l'\emph{op\'erateur de proximit\'e} de $f$.
\end{lemma}
\begin{proof}
L'existence et l'unicit\'e de $\prox_fx$ est \'etablie dans
\cite[Proposition~3.a]{More65} et la caract\'erisation dans 
\cite[Proposition~12.26]{Livre1}.
\end{proof}

\begin{example}[projection sur un convexe]
\label{ex:10}
Soient $C$ une partie convexe ferm\'ee non vide de $\HH$, 
$\iota_C$ la \emph{fonction indicatrice} de $C$, \emph{i.e.},
\begin{equation}
\label{e:iota}
(\forall y\in\HH)\quad\iota_C(y)=
\begin{cases}
0,&\text{si}\;\;y\in C;\\
\pinf,&\text{si}\;\;y\notin C,
\end{cases}
\end{equation}
et $x\in\HH$. Alors $\iota_C\in\Gamma_0(\HH)$ et nous d\'eduisons
du Lemme~\ref{l:1} qu'il existe un unique point dans $C$, not\'e
$\proj_Cx$ et appel\'e la \emph{projection} de $x$ sur $C$, qui
minimise la fonction $\delta\colon y\mapsto\|x-y\|$ sur $C$. 
Ce point est caract\'eris\'e comme suit~:
\begin{equation}
\label{e:m08}
(\forall p\in\HH)\quad\Biggl[\;p=\proj_Cx\;\;\Leftrightarrow\;\;
\begin{cases}
p\in C\\
(\forall y\in C)\quad\scal{y-p}{x-p}\leq 0.
\end{cases}
\Biggr]
\end{equation}
On appelle $\prox_{\iota_C}=\proj_C$ 
l'\emph{op\'erateur de projection} sur $C$.
\end{example}

\begin{lemma}[\protect{\cite[Proposition~12.28]{Livre1}}]
\label{l:2}
Soit $f\in\Gamma_0(\HH)$. Alors $\prox_f$ est une contraction
ferme, \emph{i.e.},
\begin{equation}
\label{e:f}
(\forall x\in\HH)(\forall y\in\HH)\quad
\|\prox_fx-\prox_fy\|^2+
\|(x-\prox_fx)-(y-\prox_fy)\|^2\leq\|x-y\|^2,
\end{equation}
et donc un op\'erateur $1$-lipschitzien.
\end{lemma}

Rappelons enfin quelques propri\'et\'es des fonctions convexes
diff\'erentiables. Tout d'abord, on dit qu'une fonction convexe
$g\colon\HH\to\RR$ est \emph{diff\'erentiable} en $x\in\HH$ s'il
existe un vecteur $\nabla g(x)\in\HH$, appel\'e le \emph{gradient}
de $g$ en $x\in\HH$, tel que 
\begin{equation}
(\forall y\in\HH)\quad\lim_{0<\alpha\downarrow 0}
\dfrac{g(x+\alpha y)-g(x)}{\alpha}=\scal{y}{\nabla g(x)}.
\end{equation}

\begin{lemma}
\label{l:5}
Soient $\beta\in\RPP$, $g\colon\HH\to\RR$ une fonction convexe
diff\'erentiable, $x\in\HH$, $y\in\HH$ et $z\in\HH$. Alors les
propri\'et\'es suivantes sont v\'erifi\'ees :
\begin{enumerate}
\item
\label{l:5i}
$g(x)\leq g(y)-\scal{y-x}{\nabla g(x)}$.
\item
\label{l:5ii}
Supposons que $\nabla g$ soit $\beta$-lipschitzien. Alors les
in\'egalit\'es suivantes ont lieu :
\begin{enumerate}
\item
\label{l:5iia}
$g(z)-\scal{z-x}{\nabla g(x)}-(\beta/2)\|z-x\|^2\leq g(x)$.
\item
\label{l:5iib}
$\scal{x-y}{\nabla g(x)-\nabla g(y)}\geq
\|\nabla g(x)-\nabla g(y)\|^2/\beta$.
\item
\label{l:5iic}
$g(z)\leq g(y)+\scal{z-y}{\nabla g(x)}+(\beta/2)\|x-z\|^2$.
\end{enumerate}
\end{enumerate}
\end{lemma}
\begin{proof}
\ref{l:5i}: \cite[Proposition~17.7(ii)]{Livre1}.

\ref{l:5iia}: 
Ce r\'esultat classique est connu sous le nom de 
<< lemme de descente >> \cite[Lemma~2.64(i)]{Livre1}.

\ref{l:5iib}: Il s'agit du << lemme de Baillon--Haddad >> 
\cite{Bail77}; \emph{cf.} \cite[Corollary~18.17]{Livre1}. 

\ref{l:5iic}: Combiner \ref{l:5i} et \ref{l:5iia}.
\end{proof}

\section{La m\'ethode}
\label{sec:3}

La m\'ethode du gradient proxim\'e est sugg\'er\'ee par la
propri\'et\'e de point fixe suivante qui caract\'erise les
solutions du Probl\`eme~\ref{prob:1}.

\begin{proposition}[\protect{\cite[Proposition~3.1(iii)]{Smms05}}]
\label{p:3}
Dans le contexte du Probl\`eme~\ref{prob:1}, soit 
$\gamma\in\RPP$ et posons 
$T=\prox_{\gamma f}\circ(\Id-\gamma\nabla g)$. Alors
\begin{equation}
\Argmin(f+g)=\menge{x\in\HH}{Tx=x}.
\end{equation}
\end{proposition}

\begin{remark}
On suppose dans le Probl\`eme~\ref{prob:1} l'existence de
solutions. Une condition suffisante pour que $\Argmin(f+g)\neq\emp$
est que $f(x)+g(x)\to\pinf$ quand $\|x\|\to\pinf$ 
\cite[Proposition~11.15(i)]{Livre1}.
\end{remark}

Sur la base de la Proposition~\ref{p:3}, le principe de
l'algorithme est de chercher un point fixe par it\'erations
successives en alternant un pas de gradient et un pas proximal.
Les propri\'et\'es asymptotiques de ce sch\'ema ont \'et\'e
\'etablies sous diverses hypoth\`eses 
dans \cite{MaPa18,Smms05,Merc80,Tsen91}. Le th\'eor\`eme suivant
regroupe les r\'esultats principaux, dont nous donnons une
d\'emonstration \'el\'ementaire.

\begin{theorem}
\label{t:1}
Dans le contexte du Probl\`eme~\ref{prob:1}, fixons $x_0\in\dom f$,
$\varepsilon\in\left]0,\beta^{-1}\right[$ et une suite
$(\gamma_n)_{n\in\NN}$ dans
$\left[\varepsilon,2\beta^{-1}-\varepsilon\right]$.
On it\`ere 
\begin{equation}
\label{e:24}
\begin{array}{l}
\text{pour}\;n=0,1,\ldots\\
\left\lfloor
\begin{array}{ll}
y_n=x_n-\gamma_n\nabla g(x_n)&{\mathrm{(un~pas~de~gradient)}}\\
x_{n+1}=\prox_{\gamma_n f}y_n&{\mathrm{(un~pas~proximal).}}\\
\end{array}
\right.\\[2mm]
\end{array}
\end{equation}
Alors les propri\'et\'es suivantes sont satisfaites, o\`u
$\mu=\min_{x\in\HH}(f(x)+g(x))$ d\'esigne la valeur optimale :
\begin{enumerate}
\item
\label{t:1i}
Convergence du gradient : Soit $x$ une solution du
Probl\`eme~\ref{prob:1}. Alors $\nabla g(x_n)\to\nabla g(x)$
avec $\sum_{n\in\NN}\|\nabla g(x_n)-\nabla g(x)\|^2<\pinf$.
\item
\label{t:1ii}
Convergence monotone en valeur :
$f(x_n)+g(x_n)\downarrow\mu$ avec
$\sum_{n\in\NN}\,(f(x_n)+g(x_n)-\mu)<\pinf$.
\item
\label{t:1iii}
Convergence de it\'er\'ees : 
La suite $(x_n)_{n\in\NN}$ converge vers une solution du
Probl\`eme~\ref{prob:1}.
\end{enumerate}
\end{theorem}
\begin{proof}
Posons $\varphi=f+g$ et prenons $y\in\HH$. 
Pour tout $n\in\NN$, puisque $x_{n+1}=\prox_{\gamma_n f}y_n$,
le Lemme~\ref{l:1} donne 
$\gamma_n^{-1}\scal{y-x_{n+1}}{y_n-x_{n+1}}+f(x_{n+1})\leq f(y)$,
d'o\`u
\begin{equation}
\label{e:59}
(\forall n\in\NN)\quad
f(x_{n+1})\leq f(y)+\scal{x_{n+1}-y}{\gamma_n^{-1}
(x_n-x_{n+1})-\nabla g(x_n)}.
\end{equation}
En outre, en vertu du Lemme~\ref{l:5}\ref{l:5iic}, 
\begin{equation}
\label{e:60}
(\forall n\in\NN)\quad
g(x_{n+1})\leq g(y)+\scal{x_{n+1}-y}{\nabla
g(x_n)}+\dfrac{\beta}{2}\|x_{n+1}-x_n\|^2.
\end{equation}
En additionnant \eqref{e:59} et \eqref{e:60} terme \`a terme, 
on obtient
\begin{align}
(\forall n\in\NN)\quad
\varphi(x_{n+1})
&\leq\varphi(y)+\dfrac{1}{\gamma_n}\scal{x_{n+1}-y}
{x_n-x_{n+1}}+\dfrac{\beta}{2}\|x_{n+1}-x_n\|^2\label{e:63}\\
&=\varphi(y)+\dfrac{1}{2\gamma_n}
\brk1{\|x_n-y\|^2-\|x_{n+1}-y\|^2}
-\dfrac{1}{2}\brk3{\dfrac{1}{\gamma_n}-\beta}\|x_{n+1}-x_n\|^2.
\label{e:66}
\end{align}
Par ailleurs, en invoquant la Proposition~\ref{p:3} et les 
Lemmes~\ref{l:2} et \ref{l:5}\ref{l:5iib}, on voit que, pour tout
$x\in\Argmin\varphi$ et tout $n\in\NN$,
\begin{align}
\label{e:87}
\|x_{n+1}-x\|^2
&=\bigl\|\prox_{\gamma_n f}\brk1{x_{n}-\gamma_n\nabla g(x_n)}-
\prox_{\gamma_n f}\brk1{x-\gamma_n\nabla g(x)}\bigr\|^2\nonumber\\
&\leq\bigl\|(x_n-x)-\gamma_n\brk1{\nabla g(x_n)-
\nabla g(x)}\bigr\|^2\nonumber\\
&=\|x_n-x\|^2-2\gamma_n\scal{x_n-x}{\nabla g(x_n)-\nabla g(x)}
+\gamma_n^2\|\nabla g(x_n)-\nabla g(x)\|^2\nonumber\\
&\leq\|x_n-x\|^2-2\beta^{-1}\gamma_n\|\nabla g(x_n)-\nabla g(x)\|^2
+\gamma_n^2\|\nabla g(x_n)-\nabla g(x)\|^2\nonumber\\
&=\|x_n-x\|^2-\gamma_n(2\beta^{-1}-\gamma_n)
\|\nabla g(x_n)-\nabla g(x)\|^2\nonumber\\
&\leq\|x_n-x\|^2-\varepsilon^2\|\nabla g(x_n)-\nabla g(x)\|^2.
\end{align}

\ref{t:1i}: On d\'eduit de \eqref{e:87} que
$(\forall N\in\NN)$ $\varepsilon^2\sum_{n=0}^N
\|\nabla g(x_n)-\nabla g(x)\|^2
\leq\|x_0-x\|^2-\|x_{N+1}-x\|^2\leq\|x_0-x\|^2$.
On fait alors tendre $N$ vers $\pinf$ pour obtenir l'assertion.

\ref{t:1ii}:
En prenant $y=x_n$ dans \eqref{e:63}, on constate que
\begin{equation}
\label{e:64}
(\forall n\in\NN)\quad\varphi(x_{n+1})
\leq\varphi(x_n)-\brk3{\dfrac{1}{\gamma_n}-\dfrac{\beta}{2}}
\|x_{n+1}-x_n\|^2
\leq\varphi(x_n)-\dfrac{\varepsilon\beta^2}{4}\|x_{n+1}-x_n\|^2
\end{equation}
et donc que
\begin{equation}
\label{e:93}
(\forall N\in\NN)\quad
\dfrac{\varepsilon\beta^2}{4}\sum_{n=0}^N
\|x_{n+1}-x_n\|^2\leq\varphi(x_0)-\varphi(x_{N+1})
\leq\varphi(x_0)-\mu.
\end{equation}
Il s'ensuit en faisant tendre $N$ vers $\pinf$ que
\begin{equation}
\label{e:61}
\sum_{n\in\NN}\|x_{n+1}-x_n\|^2\leq
\dfrac{4\brk1{\varphi(x_0)-\mu}}{\varepsilon\beta^2}<\pinf. 
\end{equation}
Par ailleurs, en prenant $y\in\Argmin\varphi$ dans \eqref{e:66}, on
d\'eduit de \eqref{e:87} que, pour tout $n\in\NN$,
\begin{align}
\label{e:67}
0&\leq\varphi(x_{n+1})-\mu\nonumber\\
&\leq\dfrac{1}{2\gamma_n}
\brk1{\|x_n-y\|^2-\|x_{n+1}-y\|^2}
-\dfrac{1}{2}\brk3{\dfrac{1}{\gamma_n}-\beta}\|x_{n+1}-x_n\|^2
\nonumber\\
&\leq
\begin{cases}
\dfrac{1}{2\varepsilon}\brk1{\|x_n-y\|^2-\|x_{n+1}-y\|^2},
&\text{si}\;\;\gamma_n\leq\dfrac{1}{\beta};\\
\dfrac{1}{2\varepsilon}\brk1{\|x_n-y\|^2-\|x_{n+1}-y\|^2}
+\dfrac{\beta}{2}\|x_{n+1}-x_n\|^2,
&\text{si}\;\;\gamma_n>\dfrac{1}{\beta}.
\end{cases}
\end{align}
En observant que 
$\sum_{n\in\NN}(\|x_n-y\|^2-\|x_{n+1}-y\|^2)\leq\|x_0-y\|^2<\pinf$
et en faisant appel \`a \eqref{e:61}, on obtient
$\sum_{n\in\NN}(\varphi(x_n)-\mu)<\pinf$. Par suite, \eqref{e:64}
garantit que $\varphi(x_n)\downarrow\mu$.

\ref{t:1iii}: D'apr\`es \eqref{e:87},  $(x_n)_{n\in\NN}$ est 
born\'ee. On peut donc en extraire une sous-suite convergente,
disons $x_{k_n}\to\overline{x}$. Notons que la semicontinuit\'e
inf\'erieure de $f$ et la continuit\'e de $g$ entra\^{\i}nent la
semicontinuit\'e inf\'erieure de $\varphi$. On tire ainsi de
\ref{t:1ii} que
$\mu\leq\varphi(\overline{x})\leq\varliminf\varphi(x_{k_n})=\mu$,
d'o\`u $\overline{x}\in\Argmin\varphi$. Puisque, d'apr\`es
\eqref{e:87}, $(\|x_n-\overline{x}\|)_{n\in\NN}$ d\'ecro\^{\i}t,
on conclut que $x_n\to\overline{x}$. 
\end{proof}

\begin{remark}
\label{r:1}\
\begin{enumerate}
\item
Si $f=0$, alors \eqref{e:24} se r\'eduit \`a l'algorithme de
descente de gradient $x_{n+1}=x_n-\gamma_n\nabla g(x_n)$.
\item
Si $g=0$, alors \eqref{e:24} se r\'eduit \`a l'algorithme du point
proximal $x_{n+1}=\prox_{\gamma_n f}x_n$.
\item
\label{r:1i}
Les conclusions du Th\'eor\`eme~\ref{t:1} demeurent valides en
pr\'esence de perturbations affectant les op\'erateurs 
$\prox_{\gamma_n f}$ et $\nabla g$ dans \eqref{e:24} et dans des
espaces de Hilbert g\'en\'eraux, la convergence dans \ref{t:1iii}
\'etant alors comprise au sens de la topologie faible
\cite{MaPa18}.
\item
\label{r:1ii}
Le Th\'eor\`eme~\ref{t:1} fournit \`a la fois 
un taux de convergence sur les valeurs dans \ref{t:1ii}, puisqu'on
en d\'eduit ais\'ement que $f(x_n)+g(x_n)-\mu=o(1/n)$
\cite{MaPa18}, et la convergence des it\'er\'ees dans \ref{t:1iii}. 
\item
\label{r:1iii}
Une version inertielle de la m\'ethode du gradient proxim\'e a 
\'et\'e propos\'ee dans \cite{Bec09a} sous la forme
\begin{equation}
\begin{array}{l}
\text{pour}\;n=0,1,\ldots\\
\left\lfloor
\begin{array}{l}
y_{n}=z_n-\beta^{-1}\nabla g(z_n)\\[1mm]
x_{n+1}={\mathrm{prox}}_{\beta^{-1}f}y_n\\[2mm]
\displaystyle{t_{n+1}=\frac{1+\sqrt{4t_n^2+1}}{2}}\\[2mm]
\displaystyle{\lambda_{n} = 1+\frac{t_n-1}{t_{n+1}}}\\[3mm]
\displaystyle{z_{n+1}=x_{n}+\lambda_n(x_{n+1}-x_n)}.
\end{array}
\right.\\
\end{array}
\end{equation}
Cet algorithme atteint un taux $f(x_n)+g(x_n)-\mu=O(1/n^2)$ sur les
valeurs, mais sans garantie de convergence des it\'er\'ees
$(x_n)_{n\in\NN}$, et il n\'ecessite en outre le stockage d'une
variable suppl\'ementaire. Une variante qui garantit la convergence
des it\'er\'ees avec un taux similaire sur les valeurs est 
analys\'ee dans \cite{Cham15}. On pourra \'egalement consulter 
\cite{Atto18,Siop17} sur ce sujet.
\end{enumerate}
\end{remark}

Un cas particulier du Probl\`eme~\ref{prob:1} est celui de la
minimisation d'une fonction lisse sous contrainte. Cela revient \`a
choisir comme fonction $f$ l'indicatrice de l'ensemble des
contraintes.

\begin{problem}
\label{prob:2}
Soient $\beta\in\RPP$, $C$ une partie convexe ferm\'ee non vide 
de $\HH$ et $g\colon\HH\to\RR$ une fonction convexe
diff\'erentiable de gradient $\beta$-lipschitzien. L'objectif est
de 
\begin{equation}
\label{e:prob2}
\minimize{x\in C}{g(x)}
\end{equation}
sous l'hypoth\`ese qu'une solution existe. 
\end{problem}

Au vu de l'Exemple~\ref{ex:10}, on obtient alors la m\'ethode du
gradient projet\'e.

\begin{corollary}
\label{c:1}
Dans le contexte du Probl\`eme~\ref{prob:2}, fixons $x_0\in C$,
$\varepsilon\in\left]0,\beta^{-1}\right[$ et une suite
$(\gamma_n)_{n\in\NN}$ dans
$\left[\varepsilon,2\beta^{-1}-\varepsilon\right]$. It\'erons
\begin{equation}
\label{e:25}
\begin{array}{l}
\text{pour}\;n=0,1,\ldots\\
\left\lfloor
\begin{array}{ll}
y_n=x_n-\gamma_n\nabla g(x_n)&{\mathrm{(un~pas~de~gradient)}}\\
x_{n+1}=\proj_{C}y_n&{\mathrm{(un~pas~de~projection).}}\\
\end{array}
\right.\\[2mm]
\end{array}
\end{equation}
Alors les propri\'et\'es suivantes sont satisfaites, o\`u
$\mu=\min_{x\in C}g(x)$ d\'esigne la valeur optimale :
\begin{enumerate}
\item
\label{c:1i}
Soit $x$ une solution du
Probl\`eme~\ref{prob:2}. Alors $\nabla g(x_n)\to\nabla g(x)$
avec $\sum_{n\in\NN}\|\nabla g(x_n)-\nabla g(x)\|^2<\pinf$.
\item
\label{c:1ii}
$g(x_n)\downarrow\mu$ avec
$\sum_{n\in\NN}\,(g(x_n)-\mu)<\pinf$.
\item
\label{c:1iii}
La suite $(x_n)_{n\in\NN}$ converge vers une solution du
Probl\`eme~\ref{prob:2}.
\end{enumerate}
\end{corollary}

\section{Applications}
\label{sec:4}

\subsection{Mod\`ele g\'en\'eral avec enveloppes de Moreau}
\label{sec:41}

Cette section s'articule autour de la notion suivante due \`a
Moreau \cite{Mor63a}.

\begin{definition}
\label{d:MoYo}
Soient $h\in\Gamma_0(\HH)$ et $\rho\in\RPP$. Alors
\begin{equation}
\label{e:moyo0}
\moyo{h}{\rho}\colon\HH\to\RR\colon x\mapsto
\inf_{y\in\HH}\brk3{h(y)+\dfrac{1}{2\rho}\|x-y\|^2}
\end{equation}
est l'\emph{enveloppe de Moreau} de $h$ de param\`etre $\rho$.
\end{definition}
  
\begin{lemma}
\label{l:9}
Soient $h\in\Gamma_0(\HH)$ et $\rho\in\RPP$. Posons
$g=\moyo{h}{\rho}$. Alors les propri\'et\'es suivantes sont
satisfaites:
\begin{enumerate}
\item
\label{l:9i}
$g\colon\HH\to\RR$ est convexe et diff\'erentiable,
avec $\nabla g=(\Id-\prox_{\rho h})/\rho$.
\item
\label{l:9ii}
$\Argmin\moyo{h}{\rho}=\Argmin h$.
\end{enumerate}
\end{lemma}
\begin{proof}
\ref{l:9i}: \cite[Proposition~12.30]{Livre1} ou \cite{Mor63a}.

\ref{l:9ii}: Soit $x\in\HH$. Alors, d'apr\`es \ref{l:9i} et la
Proposition~\ref{p:3} (appliqu\'ee avec $g=0$), 
$x\in\Argmin\moyo{h}{\rho}$ $\Leftrightarrow$ 
$\nabla\moyo{h}{\rho}(x)=0$ $\Leftrightarrow$ 
$x=\prox_{\rho h}x$ $\Leftrightarrow$ $x\in\Argmin h$.
\end{proof}

\begin{example}[distance \`a un convexe]
\label{ex:62}
Soient $C$ une partie convexe ferm\'ee non vide de $\HH$
et $d_C\colon x\mapsto\inf_{y\in C}\|x-y\|$ sa fonction distance.
Alors on d\'eduit de \eqref{e:moyo0} que 
$\moyo{\iota_C}{1}\colon\HH\to\RR\colon x\mapsto
\inf_{y\in C}\|x-y\|^2/2=d_C^2(x)/2$ et du 
Lemme~\ref{l:9}\ref{l:9i} que $\nabla d_C^2/2=\Id-\proj_C$. 
\end{example}

Nous pr\'esentons \`a pr\'esent notre corollaire principal du
Th\'eor\`eme~\ref{t:1}\ref{t:1iii}. Ce type de formulation
a \'et\'e propos\'e en th\'eorie du signal \cite{Smms05}.

\begin{corollary}
\label{c:70}
Soit $f\in\Gamma_0(\HH)$ et, pour tout $k\in\{1,\ldots,p\}$, soient
$h_k\in\Gamma_0(\GG_k)$, $L_k\colon\HH\to\GG_k$ un op\'erateur
lin\'eaire non nul, $\omega_k\in\RPP$ et $\rho_k\in\RPP$.
L'objectif est de 
\begin{equation}
\label{e:70}
\minimize{x\in\HH}{f(x)+\sum_{k=1}^p\omega_k
\brk1{\moyo{h_k}{\rho_k}}(L_kx)},
\end{equation}
sous l'hypoth\`ese qu'une solution existe. Fixons
\begin{equation}
\label{e:hyp1}
x_0\in\HH,\;\varepsilon\in\left]0,\beta^{-1}\right[\;\text{et}\;
\text{une suite}\;(\gamma_n)_{n\in\NN}\;\text{dans}\;
\left[\varepsilon,2\beta^{-1}-\varepsilon\right],\;\text{o\`u}\;\;
\beta=\sum_{k=1}^p\dfrac{\omega_k\|L_k\|^2}{\rho_k},
\end{equation}
et it\'erons
\begin{equation}
\label{e:a70}
\begin{array}{l}
\text{pour}\;n=0,1,\ldots\\
\left\lfloor
\begin{array}{l}
\text{pour}\;k=1,\ldots,p\\
\begin{array}{l}
\left\lfloor
z_{n,k}=\dfrac{\omega_k}{\rho_k}
L_k^*\brk1{L_kx_n-\prox_{\rho_kh_k}(L_kx_n)}
\right.\\[4mm]
\end{array}\\
y_n=x_n-\gamma_n\displaystyle{\sum_{k=1}^p}z_{n,k}\\
x_{n+1}=\prox_{\gamma_n f}y_n.
\end{array}
\right.\\
\end{array}
\end{equation}
Alors $(x_n)_{n\in\NN}$ converge vers une solution de \eqref{e:70}.
\end{corollary}
\begin{proof}
Posons $g=\sum_{k=1}^p\omega_k
\brk1{\moyo{h_k}{\rho_k}}\circ L_k$. Alors on d\'eduit du
Lemme~\ref{l:9}\ref{l:9i} que $g$ est convexe et diff\'erentiable
avec
\begin{equation}
\nabla g\colon x\mapsto\sum_{k=1}^p\omega_k
L_k^*\circ\nabla\brk1{\moyo{h_k}{\rho_k}}\circ L_k
=\sum_{k=1}^p\dfrac{\omega_k}{\rho_k}
L_k^*\circ(\Id-\prox_{\rho_k h_k})\circ L_k.
\end{equation}
Puisque le Lemme~\ref{l:2} garantit que les op\'erateurs
$(\Id-\prox_{\rho_k h_k})_{1\leq k\leq p}$ sont 1-lipschitziens,
$\nabla g$ est $\beta$-lipschitzien. Nous constatons ainsi
que \eqref{e:70} est un cas particulier de \eqref{e:prob1} et 
que \eqref{e:a70} est un cas particulier de \eqref{e:24}. La
conclusion d\'ecoule donc du Th\'eor\`eme~\ref{t:1}\ref{t:1iii}.
\end{proof}

\subsection{Cas particuliers}
\label{sec:42}

Nous d\'ecrivons plusieurs applications du Corollaire~\ref{c:70},
dont nous adoptons les notations et les hypoth\`eses.

\begin{example}
\label{ex:31}
On consid\`ere le probl\`eme de minimisation conjoint 
\begin{equation}
\label{e:31}
\text{trouver}\;\:x\in\Argmin f\;\;\text{tel que}\;\;
(\forall k\in\{1,\ldots,p\})\;\; L_kx\in\Argmin h_k.
\end{equation}
On peut interpr\'eter \eqref{e:70} comme une relaxation de ce
probl\`eme au sens o\`u, si \eqref{e:70} admet des solutions, alors
elles co\"{\i}ncident avec celles de \eqref{e:31} (\emph{cf.} 
Lemme~\ref{l:9}\ref{l:9ii}) tandis que, si \eqref{e:70} n'admet pas
de solution, alors \eqref{e:31} fournit des solutions
g\'en\'eralis\'ees \cite[Section~8.4.2]{Acnu24}. En particulier,
l'algorithme \eqref{e:a70} r\'esout \eqref{e:31} si ce dernier
admet une solution.
\end{example}

\begin{example}
\label{ex:32}
Dans l'Exemple~\ref{ex:31}, supposons que 
\begin{equation}
f=\iota_C\quad\text{et}\quad(\forall k\in\{1,\ldots,p\})\quad
h_k=\iota_{D_k}\quad\text{et}\quad\rho_k=1,
\end{equation}
o\`u $C\subset\HH$ et, pour chaque $k\in\{1\ldots,p\}$, 
$D_k\subset\GG_k$ sont des ensembles convexes ferm\'ees et
non vides. Alors \eqref{e:31} se r\'eduit au probl\`eme
d'admissibilit\'e convexe
\begin{equation}
\label{e:32}
\text{trouver}\;\:x\in C\;\;\text{tel que}\;\;
(\forall k\in\{1,\ldots,p\})\;\; L_kx\in D_k,
\end{equation}
tandis que, compte tenu de l'Exemple~\ref{ex:62}, 
\eqref{e:70} prend la forme du probl\`eme aux moindres carr\'es
contraints
\begin{equation}
\label{e:33}
\minimize{x\in C}{\sum_{k=1}^p\omega_kd_{D_k}^2(L_kx)}.
\end{equation}
Dans ce sc\'enario, en invoquant l'Exemple~\ref{ex:10}, 
on voit que \eqref{e:a70} se r\'e\'ecrit sous la forme 
\begin{equation}
\label{e:a32}
\begin{array}{l}
\text{pour}\;n=0,1,\ldots\\
\left\lfloor
\begin{array}{l}
\text{pour}\;k=1,\ldots,p\\
\begin{array}{l}
\left\lfloor
z_{n,k}=\omega_kL_k^*\brk1{L_kx_n-\proj_{D_k}(L_kx_n)}
\right.\\[1mm]
\end{array}\\
y_n=x_n-\gamma_n\displaystyle{\sum_{k=1}^p}z_{n,k}\\
x_{n+1}=\proj_{C}y_n,
\end{array}
\right.\\
\end{array}
\end{equation}
o\`u
\begin{equation}
\label{e:hyp6}
x_0\in\HH,\;\varepsilon\in\left]0,\beta^{-1}\right[\;\text{et}\;
(\gamma_n)_{n\in\NN}\;\text{est une suite dans}\;
\left[\varepsilon,2\beta^{-1}-\varepsilon\right],\;\text{avec}\;\;
\beta=\sum_{k=1}^p\omega_k\|L_k\|^2.
\end{equation}
\end{example}

\begin{example}[\protect{\cite{Lege05}}]
\label{ex:1805}
Dans l'Exemple~\ref{ex:32}, prenons
$\HH=\RR^N$, $C=\RR^N$ et, pour chaque 
$k\in\{1,\ldots,p\}$, $\GG_k=\RR$, $L_k\colon x\mapsto a_k^\top x$,
o\`u $a_k\in\RR^N$, $\omega_k=1$ et $D_k=\{\eta_k\}$, o\`u 
$\eta_k\in\RR$. Soit $A\in\RR^{p\times N}$ la matrices dont les
lignes sont $a_1^\top$,\,\ldots,$\,a_p^\top$ et posons
$y=(\eta_k)_{1\leq k\leq p}$.
Alors \eqref{e:32} revient \`a r\'esoudre le syst\`eme lin\'eaire
$Ax=y$ et \eqref{e:33} \`a minimiser la fonction quadratique
$x\mapsto\|Ax-y\|^2$. Cette relaxation aux moindres carr\'es
d'un syst\`eme lin\'eaire remonte aux travaux de Legendre 
\cite{Lege05}.
\end{example}

\begin{example}
\label{ex:71}
Soit $\rho\in\RPP$. Dans le Corollaire~\ref{c:70}, supposons que
$f=0$ et que $(\forall k\in\{1,\ldots,p\})$ $\GG_k=\HH$, $L_k=\Id$,
$\rho_k=\rho$ et $\omega_k=1$. Ce choix d\'ebouche sur le
probl\`eme 
\begin{equation}
\label{e:71}
\minimize{x\in\HH}{\sum_{k=1}^p\brk1{\moyo{h_k}{\rho}}(x)},
\end{equation}
qui appara\^{\i}t en apprentissage f\'ed\'er\'e \cite{Path20}.
Dans ce cas, en prenant $\gamma_n\equiv\rho/p$ dans \eqref{e:a70}. 
on obtient la m\'ethode proximale barycentrique
\begin{equation}
\label{e:a71}
\begin{array}{l}
\text{pour}\;n=0,1,\ldots\\
\left\lfloor
\begin{array}{l}
x_{n+1}=\dfrac{1}{p}\displaystyle{\sum_{k=1}^p}\prox_{\rho h_k}x_n.
\end{array}
\right.\\
\end{array}
\end{equation}
\end{example}

\begin{example}
\label{ex:72}
Soient $\rho\in\RPP$ et $h\in\Gamma_0(\HH)$. 
Dans le Corollaire~\ref{c:70} supposons que
$p=1$, $\GG_1=\HH$, $L_1=\Id$, $h_1=h$, $\rho_1=\rho$ et
$\omega_1=1$. Alors l'objectif est de 
\begin{equation}
\label{e:72}
\minimize{x\in\HH}{f(x)+\brk1{\moyo{h}{\rho}}(x)}.
\end{equation}
Dans ce cas, en prenant $\gamma_n\equiv\rho$ dans \eqref{e:a70},
on obtient la m\'ethode proximale altern\'ee
\begin{equation}
\label{e:a72}
\begin{array}{l}
\text{pour}\;n=0,1,\ldots\\
\left\lfloor
\begin{array}{l}
x_{n+1}=\prox_{\rho f}\brk1{\prox_{\rho h}x_n}.
\end{array}
\right.\\
\end{array}
\end{equation}
En particulier, supposons que $\rho=1$, $f=\iota_C$ et $h=\iota_D$,
o\`u $C$ et $D$ sont des parties convexes ferm\'ees non vides de
$\HH$. Alors, en s'appuyant sur l'Exemple~\ref{ex:62}, on voit que
\eqref{e:72}
revient \`a minimiser sur $C$ la distance \`a $D$. Par ailleurs, 
au vu de l'Exemple~\ref{ex:10}, \eqref{e:a72} donne lieu \`a la
m\'ethode des projections altern\'ees \cite{Chen59}
\begin{equation}
\label{e:a72'}
\begin{array}{l}
\text{pour}\;n=0,1,\ldots\\
\left\lfloor
\begin{array}{l}
x_{n+1}=\proj_{C}\brk1{\proj_{D}x_n}.
\end{array}
\right.\\
\end{array}
\end{equation}
\end{example}

\begin{example}
\label{ex:73}
Soient $f\in\Gamma_0(\HH)$, $\ell\in\Gamma_0(\HH)$, $z\in\HH$
et $\rho\in\RPP$. Nous nous int\'eressons au probl\`eme bivari\'e
avec couplage quadratique
\begin{equation}
\label{e:bb4}
\minimize{x\in\HH, w\in\HH}{f(x)+\ell(w)+\dfrac{1}{2\rho}
\|x+w-z\|^2},
\end{equation}
que l'on retrouve dans divers travaux 
\cite{Acke80,Atto08,Aujo06,Smms05,Merc80}.
Posons $h\colon y\mapsto \ell(z-y)$, de sorte qu'en effectuant le
changement de variable $y=z-w$, \eqref{e:bb4} devient
\begin{equation}
\label{e:bB4}
\minimize{x\in\HH, y\in\HH}{f(x)+h(y)+\dfrac{1}{2\rho}
\|x-y\|^2}.
\end{equation}
Autrement dit, nous revenons au probl\`eme \eqref{e:72} par
rapport \`a la variable $x$. On applique l'algorithme 
\eqref{e:a72} en notant que
$\prox_{\rho h}\colon x\mapsto z-\prox_{\rho\ell}(z-x)$
\cite[Proposition~24.8]{Livre1}, ce qui nous donne 
\begin{equation}
\label{e:algobb4}
\begin{array}{l}
\text{pour}\;n=0,1,\ldots\\
\left\lfloor
\begin{array}{l}
x_{n+1}=\prox_{\rho f}\brk1{x_n-z+\prox_{\rho\ell}(z-x_n)}.
\end{array}
\right.\\
\end{array}
\end{equation}
On obtient alors la convergence de $(x_n)_{n\in\NN}$ vers un point
$x\in\HH$ tel que $(x,\prox_{\rho h}(z-x))$ r\'esout \eqref{e:bb4},
si ce dernier admet une solution. 
\end{example}

\subsection{Mod\`eles avec terme quadratique}
\label{sec:43}

Dans cette section, $y$ d\'esigne un vecteur dans $\GG$,
$L\colon\HH\to\GG$ un op\'erateur lin\'eaire non nul et
$L^*\colon\GG\to\HH$ son adjoint. Le terme lisse du
Probl\`eme~\ref{prob:1} est la fonctionnelle des moindres carr\'es
\begin{equation}
g\colon x\mapsto\dfrac{1}{2}\|Lx-y\|^2.
\end{equation}
On rappelle que $g$ est convexe et diff\'erentiable et que 
$\nabla g\colon x\mapsto L^*(Lx-y)$ est $\|L\|^2$-lipschitzien.
Dans ce cadre, le Th\'eor\`eme~\ref{t:1}\ref{t:1iii} nous donne
imm\'ediatement le r\'esultat suivant.

\begin{example}
\label{ex:51}
Soit $f\in\Gamma_0(\HH)$. L'objectif est de 
\begin{equation}
\label{e:p51'}
\minimize{x\in\HH}{f(x)+\dfrac{1}{2}\|Lx-y\|^2}
\end{equation}
sous l'hypoth\`ese qu'une solution existe. 
Ce type de probl\`eme se manifeste notamment en science des
donn\'ees, o\`u $y=L\overline{x}$ repr\'esente une observation
lin\'eaire (\'eventuellement bruit\'ee) du vecteur $\overline{x}$
\`a estimer et $f$ p\'enalise une propri\'et\'e connue de
$\overline{x}$ \cite{Smms05}. Soient 
\begin{equation}
\label{e:hyp}
x_0\in\HH,\;\varepsilon\in\left]0,\|L\|^{-2}\right[\;\text{et}\;
(\gamma_n)_{n\in\NN}\;\text{une suite dans}\;
\left[\varepsilon,2\|L\|^{-2}-\varepsilon\right].
\end{equation}
It\'erons
\begin{equation}
\label{e:51a}
\begin{array}{l}
\text{pour}\;n=0,1,\ldots\\
\left\lfloor
\begin{array}{l}
z_n=\gamma_nL^*(Lx_n-y)\\
x_{n+1}=\prox_{\gamma_n f}(x_n-z_n).
\end{array}
\right.\\
\end{array}
\end{equation}
Alors $(x_n)_{n\in\NN}$ converge vers une solution de
\eqref{e:p51'}.
\end{example}

Les Exemples~\ref{ex:52'}--\ref{ex:57} ci-dessous sont des cas 
particuliers de l'Exemple~\ref{ex:51}, dont nous adoptons les 
notations et les hypoth\`eses.

\begin{example}
\label{ex:52'}
Soient $(e_k)_{1\leq k\leq N}$ une base orthonormale de $\HH$
et $(\phi_k)_{1\leq k\leq N}$ des fonctions dans $\Gamma_0(\RR)$.
L'objectif est de 
\begin{equation}
\label{e:p52'}
\minimize{x\in\HH}{\sum_{k=1}^N\phi_k\bigl(\scal{x}{e_k}\bigr)+
\dfrac{1}{2}\|Lx-y\|^2}
\end{equation}
sous l'hypoth\`ese qu'une solution existe. 
Ici, $\phi_k$ p\'enalise le $k^{\text{e}}$ coefficient de la
d\'ecomposition de $\overline{x}$ dans la base orthonormale
$(e_k)_{1\leq k\leq N}$ \cite{MaPa18,Smms05,Daub04}. Sous
l'hypoth\`ese \eqref{e:hyp}, it\'erons
\begin{equation}
\label{e:52'}
\begin{array}{l}
\text{pour}\;n=0,1,\ldots\\
\left\lfloor
\begin{array}{l}
z_n=\gamma_nL^*(Lx_n-y)\\
x_{n+1}=\sum_{k=1}^N\bigl(\prox_{\gamma_n\phi_k}
\scal{x_n-z_n}{e_k}\bigr)e_k.
\end{array}
\right.\\
\end{array}
\end{equation}
Alors $(x_n)_{n\in\NN}$ converge vers une solution de
\eqref{e:p52'}.
\end{example}
\begin{proof}
Il s'agit d'une r\'ealisation de l'Exemple~\ref{ex:51} pour
laquelle $f\colon x\mapsto\sum_{k=1}^N\phi_k(\scal{x}{e_k})$.
En effet, $f\in\Gamma_0(\HH)$ et 
$\prox_{\gamma_n f}\colon x\mapsto\sum_{k=1}^N
(\prox_{\gamma_n\phi_k}\scal{x}{e_k})e_k$
\cite[Example~2.19]{Smms05}. 
\end{proof}

\begin{example}
\label{ex:53'}
Consid\'erons le cas particulier de l'Exemple~\ref{ex:52'} o\`u
$\HH=\RR^N$, $\GG=\RR^M$, $L$ est une matrice de taille 
$M\times N$, $(e_k)_{1\leq k\leq N}$ est la base canonique de
$\RR^N$ et $(\forall k\in\{1,\ldots,N\})$ $\phi_k=|\cdot|$. Ce
choix est destin\'e \`a promouvoir la parcimonie des solutions.
Alors \eqref{e:p52'} se r\'eduit au probl\`eme << Lasso >>
\cite{Chen01,Tibs96}
\begin{equation}
\label{e:57}
\minimize{x\in\HH}{\|x\|_1+\dfrac{1}{2}\|Lx-y\|^2}
\end{equation}
et \eqref{e:52'} \`a l'algorithme de seuillage doux it\'eratif
\cite{Smms05,Daub04}
\begin{equation}
\label{e:58}
\begin{array}{l}
\text{pour}\;n=0,1,\ldots\\
\left\lfloor
\begin{array}{l}
z_n=\gamma_nL^\top(Lx_n-y)\\
x_{n+1}=\bigl(\soft_{\gamma_n}([x_n-z_n]_1),
\ldots,\soft_{\gamma_n}([x_n-z_n]_N)\bigr),
\end{array}
\right.\\
\end{array}
\end{equation}
o\`u $[x_n-z_n]_k$ d\'esigne la $k^\text{e}$ composante du vecteur
$x_n-z_n$ et
\begin{equation}
\soft_{\gamma_n}\colon\xi\mapsto
\operatorname{signe}(\xi)\max\{|\xi|-\gamma_n,0\}.
\end{equation}
\end{example}

\begin{example}
\label{ex:59}
Consid\'erons la variante dite du << filet \'elastique >>
\cite{Demo09} de l'Exemple~\ref{ex:53'} qui consiste \`a ajouter un
terme quadratique, \`a savoir, 
\begin{equation}
\label{e:99}
\minimize{x\in\HH}{\|x\|_1+\dfrac{\beta}{2}\|x\|^2
+\dfrac{1}{2}\|Lx-y\|^2}.
\end{equation}
Dans ce cas, \eqref{e:52'} donne lieu \`a l'algorithme 
\begin{equation}
\label{e:99'}
\begin{array}{l}
\text{pour}\;n=0,1,\ldots\\
\left\lfloor
\begin{array}{l}
z_n=\gamma_nL^\top(Lx_n-y)\\
x_{n+1}=\brk3{
\soft_{\frac{\gamma_n}{1+\beta\gamma_n}}
\brk3{\dfrac{[x_n-z_n]_1}{1+\beta\gamma_n}},
\ldots,\soft_{\frac{\gamma_n}{1+\beta\gamma_n}}
\brk3{\dfrac{[x_n-z_n]_N}{1+\beta\gamma_n}}}.
\end{array}
\right.\\
\end{array}
\end{equation}
\end{example}

L'exemple suivant concerne la m\'ethode de Landweber projet\'ee
\cite{Eick92}, qui est une concr\'etisation de la m\'ethode du
gradient projet\'e du Corollaire~\ref{c:1}.

\begin{example}
\label{ex:58}
Soient $C$ une partie convexe ferm\'ee non vide de $\HH$.
L'objectif est de 
\begin{equation}
\label{e:p58'}
\minimize{x\in C}{\dfrac{1}{2}\|Lx-y\|^2}
\end{equation}
sous l'hypoth\`ese qu'une solution existe. 
Sous l'hypoth\`ese \eqref{e:hyp}, it\'erons
\begin{equation}
\label{e:58a}
\begin{array}{l}
\text{pour}\;n=0,1,\ldots\\
\left\lfloor
\begin{array}{l}
z_n=\gamma_nL^*(Lx_n-y)\\
x_{n+1}=\proj_{C}(x_n-z_n).
\end{array}
\right.\\
\end{array}
\end{equation}
Alors $(x_n)_{n\in\NN}$ converge vers une solution de
\eqref{e:p58'}.
\end{example}
\begin{proof}
Il s'agit de la r\'ealisation de l'Exemple~\ref{ex:51} pour
laquelle $f=\iota_C$. On d\'eduit \eqref{e:58a} de \eqref{e:51a} en
faisant appel \`a l'Exemple~\ref{ex:10}.
\end{proof}

Si $C$ est une partie convexe de $\HH$, alors $L(C)$ est une partie
convexe de $\GG$. Le but de d'exemple suivant est de calculer la
projection sur cet ensemble.

\begin{example}
\label{ex:67}
Reprenons le cadre de l'Exemple~\ref{ex:58} en ajoutant
l'hypoth\`ese que l'ensemble $L(C)$ est ferm\'e. 
Sous l'hypoth\`ese \eqref{e:hyp}, it\'erons
\begin{equation}
\label{e:a67}
\begin{array}{l}
\text{pour}\;n=0,1,\ldots\\
\left\lfloor
\begin{array}{l}
p_n=Lx_n\\
z_n=\gamma_nL^*(p_n-y)\\
x_{n+1}=\proj_C(x_n-z_n)
\end{array}
\right.\\
\end{array}
\end{equation}
Alors $p_n\to\proj_{L(C)}y$. 
\end{example}
\begin{proof}
On a vu dans l'Exemple~\ref{ex:58} que la suite $(x_n)_{n\in\NN}$
de \eqref{e:58a} converge vers une solution $x$ de \eqref{e:p58'}. 
Posons $p=Lx$, de sorte que $p=\proj_{L(C)}y$. Alors \eqref{e:58a} 
s'\'ecrit sous la forme \eqref{e:a67} et, en invoquant la
continuit\'e de $L$, on conclut que $p_n=Lx_n\to Lx=\proj_{L(C)}y$. 
\end{proof}

\begin{example}
\label{ex:57}
Soient $(C_i)_{1\leq i\leq m}$ des parties convexes ferm\'ees
non vides de $\GG$ dont la somme de Minkowski
\begin{equation}
S=\menge{x_1+\cdots+x_m}{x_1\in C_1,\ldots,x_m\in C_m}
\end{equation}
est ferm\'ee. Le probl\`eme de projeter un point $y\in\GG$ sur $S$
se manifeste dans de nombreuses applications 
\cite{Mart94,Wang20,Wonj19}. Pour le r\'esoudre, fixons des points
$(x_{i,0})_{1\leq i\leq m}$ dans $\GG$,
$\varepsilon\in\left]0,1/m\right[$,
et une suite $(\gamma_n)_{n\in\NN}$ dans
$\left[\varepsilon,2/m-\varepsilon\right]$.
It\'erons
\begin{equation}
\label{e:a57}
\begin{array}{l}
\text{pour}\;n=0,1,\ldots\\
\left\lfloor
\begin{array}{l}
p_n=\sum_{i=1}^mx_{i,n}\\
z_n=\gamma_n(p_n-y)\\
\text{pour}\;i=1,\ldots,m\\
\left\lfloor
\begin{array}{l}
x_{i,n+1}=\proj_{C_i}(x_{i,n}-z_n).
\end{array}
\right.\\
\end{array}
\right.\\
\end{array}
\end{equation}
Alors $p_n\to\proj_Sy$. 
\end{example}
\begin{proof}
Notons $\HHH$ l'espace produit $\GG^m$ muni de la
structure euclidienne usuelle et par 
$\boldsymbol{\mathsf{x}}=(x_1,\ldots,x_m)$ un point g\'en\'erique 
dans $\HHH$. Posons $\boldsymbol{\mathsf{L}}
\colon\HHH\to\GG\colon\boldsymbol{\mathsf{x}}
\mapsto x_1+\cdots+x_m$
et $\boldsymbol{\mathsf{C}}=C_1\times\cdots\times C_m$. Alors 
$\|\boldsymbol{\mathsf{L}}\|^2=m$, 
$\boldsymbol{\boldsymbol{L}}^*\colon\GG\to\HHH\colon 
x\mapsto(x,\ldots,x)$ et
$\proj_{\boldsymbol{\mathsf{C}}}\boldsymbol{\mathsf{x}}=
(\proj_{C_1}x_1,\ldots,\proj_{C_m}x_m)$
\cite[Proposition~29.3]{Livre1}. De plus,
$\boldsymbol{\mathsf{L}}(\boldsymbol{\mathsf{C}})=S$. Ainsi,
en appliquant \eqref{e:a67} dans $\HHH$ \`a 
$\boldsymbol{\mathsf{L}}$ et $\boldsymbol{\mathsf{C}}$ avec le
point initial $\boldsymbol{\mathsf{x}}_0=(x_{1,0},\ldots,x_{m,0})$
on obtient \eqref{e:a57} et on d\'eduit la convergence de 
$(p_n)_{n\in\NN}$ de l'Exemple~\ref{ex:67}.
\end{proof}

\begin{remark}
Pour simplifier l'exposition, l'Exemple~\ref{ex:51} ne fait
intervenir qu'un seul terme quadratique. On peut ais\'ement
l'\'etendre, ainsi que les Exemples~\ref{ex:52'}--\ref{ex:57} qui
en d\'ecoulent, au probl\`eme 
\begin{equation}
\label{e:p56}
\minimize{x\in\HH}{f(x)+
\dfrac{1}{2}\sum_{k=1}^p\omega_k\|L_kx-y_k\|^2},
\end{equation}
o\`u $y_k\in\GG_k$,
$L_k\colon\HH\to\GG_k$ est lin\'eaire et $\omega_k\in\RPP$. On
remplace alors \eqref{e:hyp} par
\begin{equation}
\label{e:hyp2}
x_0\in\HH,\;\varepsilon\in\left]0,\beta^{-1}\right[\;\text{et}\;
(\gamma_n)_{n\in\NN}\;\text{une suite dans}\;
\left[\varepsilon,2\beta^{-1}-\varepsilon\right],\;\text{o\`u}\;\;
\beta=\sum_{k=1}^p\omega_k\|L_k\|^2,
\end{equation}
et on it\`ere
\begin{equation}
\label{e:56a}
\begin{array}{l}
\text{pour}\;n=0,1,\ldots\\
\left\lfloor
\begin{array}{l}
z_n=\gamma_n\sum_{k=1}^p\omega_kL_k^*(L_kx_n-y_k)\\
x_{n+1}=\prox_{\gamma_n f}(x_n-z_n).
\end{array}
\right.\\
\end{array}
\end{equation}
\end{remark}

\section{Dualit\'e}
\label{sec:5}

Nous nous int\'eressons \`a un probl\`eme de minimisation
composite. 

\begin{problem}
\label{prob:3}
Soient $\varphi\in\Gamma_0(\HH)$, $\psi\in\Gamma_0(\GG)$, 
$z\in\HH$, $r\in\GG$ et $L\colon\HH\to\GG$ un op\'erateur
lin\'eaire non nul tel que 
$r\in\text{ir}\menge{Lx-y}{x\in\dom\varphi,\;y\in\dom\psi}$.
L'objectif est de 
\begin{equation}
\label{e:prob3}
\minimize{x\in\HH}{\varphi(x)+\psi(Lx-r)+\dfrac{1}{2}\|x-z\|^2}.
\end{equation}
On note $\overline{x}=\prox_{\varphi+\psi\circ(L\cdot-r)}z$ la 
solution unique de ce probl\`eme.
\end{problem}

Le Probl\`eme~\ref{prob:3} sort {\emph{a priori} du champ du
Probl\`eme~\ref{prob:1}. Nous allons cependant \^etre en
mesure de le r\'esoudre par une approche duale propos\'ee dans
\cite{Svva10}. Le principe de la dualit\'e est d'associer au
Probl\`eme~\ref{prob:3}, dit << primal >>, un probl\`eme dit 
<< dual >>
au sens de Fenchel--Rockafellar \cite[Section~15.3]{Livre1}. Ce
probl\`eme prend ici la forme \cite{Svva10}
\begin{equation}
\label{e:frd}
\minimize{v\in\GG}
{\moyo{(\varphi^*)}{1}(z-L^*v)+\psi^*(v)+\scal{v}{r}},
\end{equation}
o\`u 
\begin{equation}
\varphi^*\colon\HH\to\RX\colon
u\mapsto\sup_{x\in\HH}\brk1{\scal{x}{u}-\varphi(x)}
\end{equation}
est la fonction conjugu\'ee de $\varphi^*$ et 
$\moyo{(\varphi^*)}{1}$ son enveloppe de Moreau 
(D\'efinition}~\ref{d:MoYo}). L'id\'ee est d'appliquer la m\'ethode
du gradient proxim\'e au probl\`eme dual \eqref{e:frd} pour
construire la solution du Probl\`eme~\ref{prob:3}.

\begin{proposition}
\label{p:17}
Dans le contexte du Probl\`eme~\ref{prob:3}, fixons 
$v_0\in\dom\psi^*$, $\varepsilon\in\left]0,\|L\|^{-2}\right[$ et 
une suite $(\gamma_n)_{n\in\NN}$ dans
$\left[\varepsilon,2\|L\|^{-2}-\varepsilon\right]$.
It\'erons
\begin{equation}
\label{e:17a}
\begin{array}{l}
\text{pour}\;n=0,1,\ldots\\
\left\lfloor
\begin{array}{l}
x_n=\prox_\varphi(z-L^*v_n)\\
v_{n+1}=\prox_{\gamma_n\psi^*}\brk1{v_n+\gamma_n(Lx_n-r)}.
\end{array}
\right.\\[2mm]
\end{array}
\end{equation}
Alors les propri\'et\'es suivantes sont satisfaites:
\begin{enumerate}
\item
\label{p:17i}
$(v_n)_{n\in\NN}$ converge vers une solution $\overline{v}$ du
probl\`eme \eqref{e:frd} et
$\overline{x}=\prox_\varphi(z-L^*\overline{v})$.
\item
\label{p:17ii}
$(x_n)_{n\in\NN}$ converge vers la solution $\overline{x}$ du
Probl\`eme~\ref{prob:3}.
\end{enumerate}
\end{proposition}
\begin{proof}
Nous donnons seulement les \'etapes principales (\emph{cf.} 
\cite{Svva10} pour les d\'etails):
\begin{itemize}
\item
On ram\`ene \eqref{e:frd} au Probl\`eme~\ref{prob:1} en
posant $f\colon v\mapsto\psi^*(v)+\scal{v}{r}$ et 
$g\colon v\mapsto{\moyo{(\varphi^*)}{1}}(z-L^*v)$. 
\item
$f\in\Gamma_0(\GG)$ et $\prox_{\gamma f}\colon v\mapsto
\prox_{\gamma\psi^*}(v-\gamma r)$.
\item
On tire du Lemme~\ref{l:9}\ref{l:9i} que $g$ est convexe et 
diff\'erentiable sur $\GG$ et que son gradient
$\nabla g\colon v\mapsto -L(\prox_{\varphi}(z-L^*v))$
est $\beta$-lipschitzien avec $\beta=\|L\|^2$.
\item
L'algorithme \eqref{e:17a} est donc un cas particulier de 
l'algorithme \eqref{e:24}. \`A ce titre le 
Th\'eor\`eme~\ref{t:1}\ref{t:1iii}
garantit que la suite $(v_n)_{n\in\NN}$ converge vers une solution
$\overline{v}$ de \eqref{e:frd}.
\item
On montre que, si $\overline{v}$ r\'esout \eqref{e:frd}, 
alors $\overline{x}=\prox_{\varphi}(z-L^*\overline{v})$.
\item
Par continuit\'e de $\prox_\varphi$ (Lemme~\ref{l:2})
et de $L^*$, $v_n\to\overline{v}$
$\Rightarrow$ $x_n=\prox_\varphi(z-L^*v_n)\to
\prox_\varphi(z-L^*\overline{v})=\overline{x}$.
\end{itemize}
\end{proof}

\begin{example}
\label{ex:90}
Dans la Proposition~\ref{p:17}, supposons que $r=0$, 
$\varphi=\iota_C$ et $\psi=\iota_D$, o\`u $C\subset\HH$ et
$D\subset\GG$ sont des parties convexes ferm\'ees telles que 
$0\in\text{ir}\,\menge{Lx-y}{x\in C,\;y\in D}$.
Alors le Probl\`eme~\ref{prob:3} est le probl\`eme de meilleure
approximation 
\begin{equation}
\label{e:2010-07-22g}
\minimize{\substack{x\in C\\ Lx\in D}}{\dfrac{1}{2}\|x-z\|^2}
\end{equation}
et son dual \eqref{e:frd} est le probl\`eme
\begin{equation}
\label{e:2010-07-22h}
\minimize{v\in\GG}
{\dfrac{1}{2}\|z-L^*v\|^2-\dfrac{1}{2} d_C^2(z-L^*v)+\sigma_D(v)},
\end{equation}
o\`u $\sigma_D\colon v\mapsto\sup_{y\in D}\scal{y}{v}$ est la
fonction d'appui de $D$. De plus, \eqref{e:17a} devient
\begin{equation}
\begin{array}{l}
\text{pour}\;n=0,1,\ldots\\
\left\lfloor
\begin{array}{l}
x_n=\proj_C(z-L^*v_n),\\
v_{n+1}=v_n+\gamma_n
\brk1{Lx_n-\proj_D(\gamma_n^{-1}v_n+Lx_n)}.
\end{array}
\right.\\[2mm]
\end{array}
\end{equation}
Soit $\overline{x}$ l'unique solution de 
\eqref{e:2010-07-22g}, \emph{i.e.}, la projection de $z$ sur
$C\cap L^{-1}(D)$. Alors $(v_n)_{n\in\NN}$ converge vers une
solution $\overline{v}$ de \eqref{e:2010-07-22h},
$\overline{x}=\proj_C(z-L^*\overline{v})$ et $(x_n)_{n\in\NN}$
converge vers $\overline{x}$.
\end{example}

\begin{example}
\label{ex:91}
Dans la Proposition~\ref{p:17}, supposons que $\psi=\sigma_D$, o\`u
$D$ est une partie convexe, compacte et non vide de $\GG$.
Alors l'objectif du Probl\`eme~\ref{prob:3} est de
\begin{equation}
\label{e:prob41}
\minimize{x\in\HH}{\varphi(x)+\sigma_D(Lx-r)+\frac12\|x-z\|^2},
\end{equation}
celui du probl\`eme dual \eqref{e:frd} est de
\begin{equation}
\label{e:prob42}
\minimize{v\in D}{\moyo{(\varphi^*)}{1}(z-L^*v)+\scal{v}{r}},
\end{equation}
et l'algorithme \eqref{e:17a} se r\'eduit \`a
\begin{equation}
\label{e:main41}
\begin{array}{l}
\operatorname{pour}\;n=0,1,\ldots\\
\left\lfloor
\begin{array}{l}
x_n=\prox_\varphi(z-L^*v_n)\\[1mm]
v_{n+1}=\proj_D\brk1{v_n+\gamma_n(Lx_n-r)}.
\end{array}
\right.\\[2mm]
\end{array}
\end{equation}
Soit $\overline{x}$ l'unique solution de 
\eqref{e:prob41}. Alors $(v_n)_{n\in\NN}$ converge vers une
solution $\overline{v}$ de \eqref{e:prob42},
$\overline{x}=\prox_\varphi(z-L^*\overline{v})$ et 
$(x_n)_{n\in\NN}$ converge vers $\overline{x}$.
Notons que \eqref{e:prob41} englobe en particulier des 
probl\`emes en d\'ebruitage de signaux \cite{Aube06,Svva10} et en 
m\'ecanique \cite{Ekel99,Merc80}.
\end{example}

\section{Une version multivari\'ee}
\label{sec:6}

Certains probl\`emes d'optimisation font intervenir $m$
variables $x_1\in\HH_1$, \ldots, $x_m\in\HH_m$ interagissant entre
elles. Ces formulations se retrouvent par exemple en 
d\'ecomposition de domaine \cite{Atto08,Sico10}, 
en apprentissage automatique \cite{Bach12,McDo16}, 
en traitement du signal et de l'image \cite{Bric09,Chau13}, 
et dans les probl\`emes de flots \cite{Rock95}. 
Nous suivons l'approche de \cite{Sico10}.

\begin{problem}
\label{prob:62}
Soient $(\HH_i)_{1\leq i\leq m}$ et $(\GG_k)_{1\leq k\leq p}$ des
espaces euclidiens. Pour tout $i\in\{1,\ldots,m\}$, soit 
$f_i\in\Gamma_0(\HH_i)$ et, pour tout $k\in\{1,\ldots,p\}$, soient 
$\tau_k\in\RPP$, $h_k\colon\GG_k\to\RR$ une fonction convexe
et diff\'erentiable de gradient $\beta_k$-lipschitzien, et  
$L_{ki}\colon\HH_i\to\GG_k$ un op\'erateur lin\'eaire. 
On suppose que $\min_{1\leq k\leq p}\sum_{i=1}^m\|L_{ki}\|^2>0$. 
L'objectif est de 
\begin{equation}
\label{e:genna07-4}
\minimize{x_1\in\HH_1,\ldots,\,x_m\in\HH_m}{\sum_{i=1}^mf_i(x_i)
+\sum_{k=1}^ph_k\bigg(\sum_{i=1}^mL_{ki}x_i\bigg)}
\end{equation}
sous l'hypoth\`ese qu'une solution existe. 
\end{problem}

Dans \eqref{e:genna07-4}, le terme s\'eparable 
$\sum_{i=1}^mf_i(x_i)$ p\'enalise les composantes 
$(x_i)_{1\leq i\leq m}$ individuellement, tandis que le terme 
$\sum_{k=1}^ph_k(\sum_{i=1}^mL_{ki}x_i)$ p\'enalise $p$ couplages
entre ces variables mod\'elisant les interactions. 

La strat\'egie est de se ramener au Probl\`eme~\ref{prob:1} en
posant
\begin{equation}
\begin{cases}
\HH=\HH_1\oplus\cdots\oplus\HH_m\\
f\colon\HH\to\RX\colon(x_1,\ldots,x_m)
\mapsto\displaystyle\sum_{i=1}^mf_i(x_i)\\
g\colon\HH\to\RR\colon(x_1,\ldots,x_m)\mapsto
\displaystyle\sum_{k=1}^ph_k
\brk3{\displaystyle\sum_{i=1}^mL_{ki}x_i}.
\end{cases}
\end{equation}
On note en effet que :
\begin{itemize}
\item
$f\in\Gamma_0(\HH)$ et, pour tout $\gamma\in\RPP$,
\begin{equation}
\prox_{\gamma f}\colon(x_1,\ldots,x_m)\mapsto
\brk2{\prox_{\gamma f_1}x_1,\ldots,\prox_{\gamma f_m}x_m}.
\end{equation}
\item
$g$ est convexe et lisse, et son gradient
\begin{equation}
\nabla g\colon(x_1,\ldots,x_m)\mapsto
\brk4{\sum_{k=1}^pL_{k1}^*\brk3{\nabla h_k
\brk3{\sum_{j=1}^mL_{kj}x_j}},
\ldots,
\sum_{k=1}^pL_{km}^*\brk3{\nabla h_k
\brk3{\sum_{j=1}^mL_{kj}x_j}}}
\end{equation}
a constante de Lipschitz
\begin{equation}
\label{e:97}
\beta={p\:\displaystyle{\max_{1\leq k\leq p}}
\:\tau_k\sum_{i=1}^m\|L_{ki}\|^2}.
\end{equation}
\end{itemize}
On d\'eduit alors du Th\'eor\`eme~\ref{t:1}\ref{t:1iii} le
r\'esultat suivant.

\begin{proposition}
\label{p:99}
Dans le contexte du Probl\`eme~\ref{prob:62} et de \eqref{e:97}, 
fixons $x_{1,0}\in\dom f_1, \ldots, x_{m,0}\in\dom f_m$,
$\varepsilon\in\left]0,\beta^{-1}\right[$ et une suite
$(\gamma_n)_{n\in\NN}$ dans
$\left[\varepsilon,2\beta^{-1}-\varepsilon\right]$.
On it\`ere 
\begin{equation}
\label{e:98}
\begin{array}{l}
\text{pour}\;n=0,1,\ldots\\
\left\lfloor
\begin{array}{l}
\text{pour}\;i=1,\ldots,m\\
\begin{array}{l}
\left\lfloor
\begin{array}{l}
y_{i,n}=x_{i,n}-\gamma_{n}\displaystyle{\sum_{k=1}^p}
L_{ki}^*\brk4{\nabla h_k\brk3{\sum_{j=1}^mL_{kj}x_{j,n}}}\\
x_{i,n+1}=\prox_{\gamma_n f_i}y_{i,n}.
\end{array}
\right.
\end{array}
\end{array}
\right.\\[2mm]
\end{array}
\end{equation}
Alors, pour tout $i\in\{1,\ldots,m\}$, la suite 
$(x_{i,n})_{n\in\NN}$ converge vers un point $x_i\in\HH_i$, et
$(x_1,\ldots,x_m)$ r\'esout le Probl\`eme~\ref{prob:62}.
\end{proposition}

On trouvera dans \cite{Bric09} des applications de la
Proposition~\ref{p:99} au traitement du signal, dont celle-ci.

\begin{example} 
\label{ex:95} 
Consid\'erons, dans le cadre du Probl\`eme~\ref{prob:62}, celui de
reconstruire un signal multi-canaux $(x_1,\ldots,x_m)$ \`a partir
de $p$ observations impr\'ecises 
\begin{equation} 
(\forall k\in\{1,\ldots,p\})\quad y_k\approx\sum_{i=1}^mL_{ki}x_i
\end{equation} 
de m\'elanges lin\'eaires, sous la contrainte que chaque composante
$x_i$ appartienne \`a un convexe ferm\'e non vide $C_i$ de
$\HH_i$. On mod\'elise ce probl\`eme sous la forme 
\begin{equation}
\label{e:genna07-5} 
\minimize{x_1\in C_1,\ldots,\,x_m\in C_m}{\dfrac{1}{2}
\sum_{k=1}^p\Biggl\|y_k-\sum_{i=1}^mL_{ki}x_i\Biggr\|^2},
\end{equation} 
ce qui revient \`a poser $f_i=\iota_{C_i}$ et
$h_k\colon v_k\mapsto\|y_k-v_k\|^2/2$ dans le
Probl\`eme~\ref{prob:62}. L'algorithme \eqref{e:98} s'\'ecrit alors
\begin{equation}
\label{e:95}
\begin{array}{l}
\text{pour}\;n=0,1,\ldots\\
\left\lfloor
\begin{array}{l}
\text{pour}\;i=1,\ldots,m\\
\begin{array}{l}
\left\lfloor
\begin{array}{l}
y_{i,n}=x_{i,n}+\gamma_{n}\displaystyle{\sum_{k=1}^p}
L_{ki}^*\brk3{y_k-\sum_{j=1}^mL_{kj}x_{j,n}}\\
x_{i,n+1}=\proj_{C_i}y_{i,n}
\end{array}
\right.
\end{array}
\end{array}
\right.
\end{array}
\end{equation}
et sa convergence vers une solution de \eqref{e:genna07-5} est
garantie par la Proposition~\ref{p:99}. Notons que dans le
sc\'enario pr\'esent on peut affiner \eqref{e:97} en prenant
$\beta=\sum_{k=1}^p\sum_{i=1}^m\|L_{ki}\|^2$.
\end{example}

\section{Conclusion}
\label{sec:7}

L'objet principal de cette synth\`ese a \'et\'e de montrer que,
malgr\'e son formalisme simple, le Probl\`eme~\ref{prob:1}
mod\'elise une grande vari\'et\'e de probl\`emes concrets et qu'il
peut \^etre r\'esolu par un algorithme alternant un pas de gradient
sur sa fonction lisse et un pas proximal sur sa fonction non
diff\'erentiable~: la m\'ethode du gradient proxim\'e. Les
propri\'et\'es asymptotiques de cet algorithme ont \'et\'e
\'etudi\'ees et plusieurs de ses applications ont \'et\'e
d\'ecrites. Enfin, nous avons vu que, par le biais de
reformulations duales ou dans des espaces produits, la port\'ee de
la m\'ethode du gradient proxim\'e peut \^etre \'etendue \`a des
probl\`emes d'optimisation qui se situent au del\`a du
cadre initial du Probl\`eme~\ref{prob:1}. 

\bigskip

\noindent
{\bfseries Remerciements.} L'auteur remercie Minh N. B\`ui, 
Diego J. Cornejo et Julien N. Mayrand pour leurs relectures 
attentives.

\end{document}